\documentclass[12pt,a4paper]{article}
\usepackage{amsmath,amsthm,amssymb,latexsym,epic}

\newtheorem{theorem}{Theorem}[section]
\newtheorem{lemma}[theorem]{Lemma}
\newtheorem{corollary}[theorem]{Corollary}
\newtheorem{proposition}[theorem]{Proposition}
\newtheorem{example}[theorem]{Example}
\newtheorem{remark}[theorem]{Remark}
\usepackage[all]{xy}
\usepackage{enumerate}

\newcommand{\N}{\mathbb{N}}

\newcommand{\IS}{\mathcal{IS}}

\renewcommand{\S}{\mathcal{S}}

\newcommand{\rank}{\mathrm{rank}}
\newcommand{\ran}{\mathrm{ran}}

\newcommand{\dom}{\mathrm{dom}}

\renewcommand{\phi}{\varphi}

\begin{document}

\title{Semitransitive and transitive subsemigroups of the  inverse symmetric semigroups
\thanks{The authors were supported in part by  Ukrainian-Slovenian bilateral
research grants from the Ministry of Education and Science,
Ukraine, and the Research Agency of the Republic of Slovenia.}}
\author{Karin Cvetko-Vah, Damjana Kokol Bukov\v{s}ek, Toma\v{z}
Ko\v{s}ir, \\ Ganna Kudryavtseva, Yaroslav Lavrenyuk and Andriy
Oliynyk}
\date{\today}
\maketitle

\begin{abstract} We classify minimal transitive subsemigroups of
the finitary inverse symmetric semigroup modulo the classification
of minimal transitive subgroups of  finite symmetric groups; and
semitransitive subsemigroups of the finite inverse symmetric
semigroup of the minimal cardinality modulo the classification of
transitive subgroups of the minimal cardinality of  finite
symmetric groups.
\end{abstract}
%{\large
\section{Introduction}

An action of a semigroup $S$ on a set $X$ is said to be {\em
transitive} if for every odered pair $(x,y)$ in $X\times X$ there
is an element of $S$ that maps $x$ to $y$. Recently, a weaker
notion of semitransitivity was introduced by Rosenthal and
Troitsky \cite{RT}. An action of $S$ on $X$ is {\em
semitransitive} if for every ordered pair $(x,y)$ in $X\times X$
there is an element $\varphi$ in $S$ such that either $x=y\varphi$
or $y=x\varphi$. So far the research of semitransitivity was
mostly focused on the linear case, where $X$ is a vector space and
$S$ consists of linear maps. Both can have some additional
structure. For example, Rosenthal and Troitsky considered
subalgebras of bounded linear operators on a Banach space.
Semitransitive actions on a vector space of algebras and
semigroups of linear operators were considered in \cite{BGMRT}, of
vector spaces of linear operators in \cite{BDKKO,RaT,Bled} and of
Jordan algebras in \cite{BDKKOR}. Among other things, minimal
semitransitive algebras were characterized in \cite{RT} as those
simultaneously similar to the algebra of upper triangular Toeplitz
operators (that is, the algebra generated by the identity and a
nilpotent matrix of index $n$). Semitransitive actions of vector
spaces $S$ of linear operators on a vector space $X$ were studied
also in \cite{RaT}, where $k$-fold semitransitivity was also
considered and Jacobson's density theorem for rings was extended.
Examples of minimal semitransitive subspaces of linear maps acting
on a finite-dimensional vector space that have no trivial
invariant subspaces are given in \cite{Bled} and in \cite{BDKKO}
it is proved that if the dimensions of a vector spaces $S$ and $X$
are equal then $S$ is triangularizable.

In this paper, we study semigroups of partial transformations of a
set $X$ (see \cite{GM}). A semigroup of partial transformations of
a set $X$ is called {\em transitive} provided that for every
ordered pair $(x,y)\in X\times X$ there is some $\varphi\in S$
such that $x\varphi=y$, and it is {\em semitransitive} provided
that for every $(x,y)\in X\times X$ there is some $\varphi\in S$
such that either $x\varphi=y$ or $y\varphi=x$. The definitions
ensure that transitivity of $S$ implies its semitransitivity. If
$S$ is an inverse semigroup (in particular, if it is a group) then
the converse is also true. However, in general, there are
semitransitive semigroups which are not transitive.

We note that Schein \cite{schein} proved a number of results on transitive
effective representations of inverse semigroups by partial
one-to-one transformations of sets. Though to the best of our
knowledge transitive and semitransitive semigroups of partial
one-to-one transformations themselves have not been dealt with in
the literature.

The aim of our paper is to give a classification of minimal
transitive subsemigroups of the finitary inverse symmetric
semigroup modulo the classification of minimal transitive
subgroups of  finite symmetric groups and the classification of
semitransitive subsemigroups of the finite inverse symmetric
semigroup of the minimal cardinality modulo the classification of
transitive subgroups of the minimal cardinality of  finite
symmetric groups. It is well known that the classification of
minimal transitive subgroups of finite symmetric groups is a
difficult task. Namely, suppose that a finite group $G$ acts
faithfully and transitively on a finite set $X$ and write $|X|=n$
for the cardinality of $X$. We view $G$ as a subgroup of the
symmetric group $S_n$ of degree $n$. If $X=G$ and the action of
$G$ is by multiplication then we say that $G$ acts regularly.
%Then obviously the cardinalities  $|X|$ and $|G|$ are equal.
Note that, in general, the transitivity implies that $|G|\ge |X|$.
Igo Dak Tai \cite{tai} showed that if $n\neq p,p^2$, where $p$ is
a prime, then there exists a minimal transitive subgroup $G$ of
$S_n$ that is not regular and such that $|G|>n$. Suprunenko
\cite{supr} showed that if $G$ is a minimal solvable transitive
subgroup of $S_n$, $n=pq$, $p$ and $q$ distinct primes such that
$p>q$ and $q\not|(p-1)$, then $G$ is either cyclic of order $pq$
or a minimal nonabelian group of order $p^mq$ or $pq^l$, where $m$
is the order of $p$ in the multiplicative group of the Galois
field $GF(q)$ and $l$ the multiplicative order of $q$ in $GF(p)$.
The remaining case $q|(p-1)$ was treated by Kopylova \cite{kop1}.
Kopylova in \cite{kop2} studied the structure of those subgroups
that occur as minimal (nonregular) transitive subgroups of $S_n$.
She showed that $A_5$ occurs as a minimal transitive subgroup of
$S_{10}$. More recently, Hulpke \cite{hul} (see also his PhD
thesis \cite{hulthesis}) listed all the transitive (nonregular)
groups up to degree $n=30$.

We conclude the introduction with a brief description of the setup of the paper. In the second section
we recall the definitions needed in the sequel and give the classification of minimal transitive
subsemigroups of the finitary inverse symmetric semigroup modulo the classification
of minimal transitive subgroups of finite symmetric groups. In the third section we
classify semitransitive subsemigroups of the finite inverse symmetric
semigroup of the minimal cardinality modulo the classification of
transitive subgroups of the minimal cardinality of  finite
symmetric groups.

\section{Classification of minimal transitive sub\-se\-mi\-groups}

We start by recalling standard definitions and elementary
properties of regular and inverse semigroups which we use in the
sequel \cite{hig,law,pet}. Let $S$ be a semigroup. Two elements
$a,b\in S$ are called {\em mutually inverse} provided that $aba=a$
and $bab=b$. Whenever the stated equalities hold, we also say that
$a$ is an inverse of $b$, and $b$ is an inverse of $a$. An element
$a\in S$ is called {\em regular} provided that it possesses at
least one inverse element. For $a\in S$ to be regular it is enough
to require that there is $b\in S$ such that $aba=a$ (then $a$ and
$bab$ are mutually inverse). The semigroup $S$ is called {\em
regular} provided that every $a\in S$ is regular. Further, $S$ is
called {\em inverse} provided that every element in $S$ possesses
a unique inverse. Equivalently, $S$ is inverse if and only if it
is regular and its idempotents commute. Therefore, a regular
subsemigroup of an inverse semigroup is necessarily inverse.

By $\IS(X)$ we denote the {\em full inverse symmetric semigroup}
over an underlying set $X$. (See \cite{hig,law,pet}.) In the case
when $X$ is a finite set and $n=|X|$ we write $\IS_n$ for $\IS(X)$
and take the convention that $X=\{1,\dots, n\}$. In the case when
$X$ is infinite we assume that $X\supseteq \N$. If $X$ is infinite
then all elements of $\IS(X)$ of finite ranks form a subsemigroup
(and even an ideal). We denote this semigroup by $\IS_{fin}(X)$
and refer to it as to {\em the semigroup of all finitary partial
permutations of $X$}. Accordingly, we call an element of finite
rank a {\em finitary element}. We note that when $X$ is a finite
set, one obviously has $\IS_n=\IS_{fin}(X)$.

From now on, by the inverse element to a given $a\in \IS(X)$ we
mean the element $a^{-1}$ which is inverse to $a$ in $\IS(X)$.

\begin{proposition}\label{pr:min_tr_simple}
Let $S$ be a transitive subsemigroup of $\IS(X)$ and $I$  a
non-zero ideal of $S$.  Then $I$ is a transitive ideal of
$\IS(X)$.
\end{proposition}

\begin{proof}
Let $\varphi\in I$, $\varphi\neq 0$. Consider any
$i\in\dom\varphi$. Suppose $j=i\varphi$. Let $k,l\in X$.
Transitivity of $S$ ensures that there exist some
$\psi_1,\psi_2\in S$ satisfying $k\psi_1=i$ and $j\psi_2=l$. Then
$k\psi_1\varphi\psi_2=l$, and thus $\psi_1\varphi\psi_2\in I$.
\end{proof}

Recall that a semigroup $S$ is called {\em simple} if it does not
possess proper ideals, and {\em $0$-simple} if it has a zero
element $0$, $S^2\neq \{0\}$ and $S$ does not possess proper
non-zero ideals.

\begin{corollary}\label{cor:simple}
A minimal transitive subsemigroup $S$ of $\IS(X)$ is either $0$-simple
or simple depending on whether $S$ contains the zero element or
not.
\end{corollary}

\begin{proposition}\label{pr:min_tr}
Suppose $S$ is a transitive subsemigroup of $\IS(X)$ such that
$S\cap\IS_{fin}(X)$ contains a non-zero element. Then there is an
inverse transitive subsemigroup of $\IS(X)$ which is contained in
$S\cap\IS_{fin}(X)$.
\end{proposition}

\begin{proof}
As $S\cap \IS_{fin}(X)$ is a non-zero ideal of $S$ then by
Proposition~\ref{pr:min_tr_simple} we have that $S\cap
\IS_{fin}(X)$ is also a transitive subsemigroup of $\IS(X)$. Hence
we can assume that all elements of $S$ are finitary.

It is enough to show that $S$ contains a regular transitive
subsemigroup.

Let $i,j\in X$. First let us show that $S$ contains a pair of
elements $a_{i,j}$, $a_{i,j}^{-1}$ such that $a_{i,j}$ is mapping
$i$ to $j$. Consider an arbitrary element $\varphi\in S$ such that
$i\varphi=j$. Consider any $\psi\in S$ such that $j\psi =i$ (the
existence of such an element $\psi$ is provided by the
transitivity of $S$). Obviously,  we have $\rank
\varphi\psi\varphi\leq \rank\varphi$ and $\rank
\psi\varphi\psi\leq \rank\psi$. Consider two possible cases.

{\em Case 1.} Suppose $\rank \varphi\psi\varphi = \rank\varphi$ or
$\rank \psi\varphi\psi= \rank\psi$. Suppose that the first of
these equalities holds (if this is not the case then we just
switch $\varphi$ and $\psi$). Set
$\overline{\psi}=\psi\varphi\psi$. Then
$\rank\overline{\psi}\leq\rank\psi$ and $j\overline{\psi}=i$. Observe that
$$\rank\overline{\psi}\varphi\overline{\psi}=\rank\psi\varphi\psi\varphi\psi=
\rank\psi\varphi\psi=\rank\overline{\psi}$$
and similarly
$$\rank\varphi\overline{\psi}\varphi=\rank\varphi\psi\varphi\psi\varphi=\rank\varphi
\psi\varphi=\rank\varphi.$$
Thus we replace $\psi$ by $\overline{\psi}$ if needed and assume that
$\rank\psi\leq\rank\varphi$. This together with
$\rank \varphi\psi\varphi = \rank\varphi$ implies that
$\ran\varphi=\dom\psi$ and $\ran\psi=\dom\varphi$. If
$\varphi\psi\varphi=\varphi$, then $\psi=\varphi^{-1}$, and we are
done.  Otherwise, set $M=\dom\varphi$ and note that
$\dom\varphi\psi=\ran\varphi\psi=M$. This implies that there is
some $k\geq 1$ such that $\left(\varphi\psi\right)^k$ equals the identity
transformation of $M$. Fix such a $k$. Then we can write
$$
\varphi=(\varphi\psi)^k\varphi=\varphi\cdot(\psi\varphi)^{k-1}\psi\cdot\varphi,
$$
ensuring that $\alpha\varphi\alpha=\varphi^{-1}$, where $\alpha =
(\psi\varphi)^{k-1}\psi$.

{\em Case 2.} Suppose $\rank \varphi\psi\varphi < \rank\varphi$
and $\rank \psi\varphi\psi < \rank\psi$.  Set
$\varphi_1=\varphi\psi\varphi$ and $\psi_1=\psi\varphi\psi$. Then
$\rank\varphi_1<\rank\varphi$ and $\rank\psi_1<\rank\psi$. Note
that we still have $i\varphi_1=j$ and $j\psi_1=i$. In particular,
if $\rank\varphi=1$ or $\rank\psi=1$, this case does not occur.

Applying the argument above at most
$min\{\rank\varphi,\rank\psi\}-1$ times we either find an element
mapping $i$ to $j$, which, together with its inverse, lies in $S$,
or find an element of rank 1 mapping $i$ to $j$. But for such an
element its inverse must also belong to $S$ since for such an element
Case 2 is not possible.

For every $i,j$ in $X$, we fix some $a_{i,j}$ such
that $a_{i,j},a_{i,j}^{-1}\in S$ and $a_{i,j}$ maps $i$ to $j$. We
consider the subsemigroup $T$ of $S$ generated by all such pairs
$a_{i,j}$ and $a_{i,j}^{-1}$. It is obviously transitive.
Besides, it is regular since for any $b=b_1\cdots b_m\in T$, where
$b_1, \dots, b_m$ are generators of $T$, we know that the element
$b'=b_m^{-1}\cdots b_1^{-1}$, which also belongs to $T$, is the inverse
of $b$ in $\IS(X)$. Thus $b$ and $b'$ form a pair of mutually
inverse elements in $T$. It follows that $T$ is a regular
semigroup and the proof is complete.
\end{proof}

\begin{remark}  In
Proposition \ref{pr:min_tr} the requirement that
$S\cap\IS_{fin}(X)$ contains a non-zero element $S$ is essential.
If this requirement is not satisfied, $S$ may not contain an
inverse transitive subsemigroup as the following example
illustrates.
\end{remark}

\begin{example}{\em
Let $X$ be infinite.  Let $S$ be a subsemigroup of $\IS(X)$
consisting of all $\varphi\in\IS(X)$ such that $\dom\varphi=X$ and
$\ran\varphi\neq X$. Obviously, $S$ is transitive. Besides, $S$ is
idempotent-free. It follows that $S$ does not contain inverse
subsemigroups.}
\end{example}

In order to formulate and prove the classification of minimal
transitive subsemigroups of $\IS_{fin}(X)$ we recall the
definition of a Brandt semigroup and then establish how Brandt
subsemigroups of $\IS_{fin}(X)$ are built.

Let $G$ be a group, $0\not\in G$ some symbol, which we call zero,
and $I$ a set. Let $B(G,I)$ be the set of all matrices with
entries from $G\bigcup \{0\}$, whose rows and columns are indexed
by $I$, and which contain at most one non-zero element. We denote
the zero matrix just by $0$, and the matrix whose the only entry
from $G$ equals $g$ and is in the position $(i,j)$  by $M(i,g,j)$.
The product in $B(G,I)$ is given by
\begin{equation}\label{eq:brandt}
M(i,g,j)M(k,h,l)=\left\lbrace\begin{array}{l} M(i,gh,l), \text{ if}
j=k;\\
0, \text{ otherwise},\end{array}\right.
\end{equation}
and the product with $0$ at either side is again $0$. The
semigroup $B(G,I)$ is called a {\em Brandt semigroup} (see
\cite{hig,law,pet}). Brandt semigroups are inverse completely
$0$-simple semigroups, and every completely $0$-simple inverse
semigroup is isomorphic to some Brandt semigroup.

\begin{corollary}\label{cor:brandt}
A subsemigroup $S$ of $\IS_{fin}(X)$ is a minimal transitive
subsemigroup of $\IS_{fin}(X)$, which is not a group, if and only
if $S$ is a minimal transitive Brandt subsemigroup of
$\IS_{fin}(X)$.
\end{corollary}

\begin{proof}
The statement follows from Corollary~\ref{cor:simple},
Proposition~\ref{pr:min_tr} and the fact that a simple inverse
semigroup is actually a group.
\end{proof}

Now we determine the structure of transitive Brandt subsemigroups
of $\IS_{fin}(X)$. Let $I$ be an index set, $|I|>1$. Let, further,
$\{M_i, \, i\in I\}$ be a collection of pairwise disjoint subsets
of $X$ such that $X=\cup_{i\in I}M_i$ and all $M_i$-s are of the
same finite cardinality (this, in particular, implies that $X$ and
$I$ must have the same cardinality whenever $X$ is infinite).
Suppose $1\in I$ and let $G$ be a subgroup of $\S(M_1)$. For every
$i\in I$ fix some bijection $\pi_i$ from $M_1$ to $M_i$. For every
$i,j\in I$ and $g\in G$ set $M(i,g,j)=\pi_i^{-1}g\pi_j$.

\begin{proposition}\label{prop:brandt}
\begin{enumerate}
\item \label{i1} All $M(i,g,j)$ together with $0$ form a Brandt subsemigroup
of $\IS_{fin}(X)$, which we denote by $B(G; M_i; \pi_i)_{i\in I}$.
\item \label{i12} $B(G; M_i; \pi_i)_{i\in I} \subseteq B(H; M_i; \pi_i)_{i\in
I}$ if and only if $G<H$.
\item\label{i2} $B(G;M_i; \pi_i)_{i\in I}$ is a transitive subsemigroup of $\IS_{fin}(X)$ if and only
if $G$ is a transitive subgroup of $\S(M_1)$.
\item\label{i3} Every transitive Brandt subsemigroup of $\IS_{fin}(X)$
coincides with some $B(G;M_i;\pi_i)_{i\in I}$, where $G$ is a
transitive subgroup of $\S(M_1)$.
\end{enumerate}
\end{proposition}

\begin{proof} It follows from the definition of $B(G;M_i; \pi_i)_{i\in I}$ that the multiplication in it obeys the
rule~\eqref{eq:brandt}, and therefore~\ref{i1} is proven.

The second claim is straightforward.

Suppose $G$ is a transitive subgroup of $\S(M_1)$ and take
arbitrary $k,l\in X$. Suppose $k\in M_i$, $l\in M_j$.
We can map $k$ and $l$ to $M_1$ by $\varphi = M(i,e,1)$ and
$\psi=M(j,e,1)$ respectively. Further, $k\varphi$ can be mapped to
$l\psi$ by some $\tau=M(1,g,1)$ by the transitivity of $G$. It
follows that $\varphi\tau\psi^{-1}$ maps $k$ to $l$, implying the
transitivity of $B(G;M_i; \pi_i)_{i\in I}$. Further, the mappings
from $M_1$ to itself are given by the elements of the form
$M(1,g,1)$ where $g\in G$. Thus, if $G$ is not transitive not
every element in $M_1$ can be mapped to any element of $M_1$ by
$B(G;M_i; \pi_i)_{i\in I}$. Hence,~\ref{i2} is proven.

Now, suppose that $S$ is a  transitive Brandt subsemigroup of
$\IS_{fin}(X)$.
%Proposition~\ref{pr:min_tr} ensures us that $S$ is inverse.
Consider any element $\varphi\in S$. Suppose
$k=\rank\varphi$.
%Proposition~\ref{pr:min_tr_simple} implies that
Then all the non-zero elements in $S$ must also have rank $k$
(since otherwise all elements of the smallest positive rank in $S$
would form a transitive non-zero ideal). Denote $M_1=\dom\varphi$
and fix $1\in M_1$. We note that for every element in $S$ its
domain and range should either coincide with $M_1$ or be disjoint
with it, as otherwise we would be able to find elements with
positive ranks strictly less than $k$. Let $i\in X$. There is
$\pi$ in $S$ mapping $1$ to $i$. Since $\dom\pi\cap
M_1\neq\varnothing$ we have that $\dom\pi=M_1$. Similarly we
conclude that $\ran\pi=M_1$ if $i\in M_1$, and $\ran\pi\cap
M_1=\varnothing$ if $i\not\in M_1$. This argument ensures that
there is a decomposition $X=\cup_{i\in I}M_i$ such that $M_i$-s
are pairwise disjoint and all $M_i$-s are of the same finite
cardinality. And, moreover, for every $i\in I$ there is some
$\pi_i\in S$ with $\dom\pi_i=M_1$ and $\ran\pi_i=M_i$.

To map an element of $M_1$ to another element of $M_1$ we have to
act by some $\pi$ with $\dom\pi=\ran\pi=M_1$, that is, by an
element of a maximal subgroup $\S(M_1)$. It follows that
$S\cap\S(M_1)$ is a transitive subgroup of $\S(M_1)$ which we
denote by $G$.

We know that $S$ contains $G$, and all $\pi_i$, $i\in I$. Consider
the semigroup $S'$, generated as an inverse semigroup by $G$ and
$\pi_i$, $i\in I$. It equals $B(G;M_i;\pi_i)_{i\in I}$ and thus,
in particular, is transitive. We are left to conclude that it must
coincide with $S$ in view of the minimality of $S$.
Therefore,~\ref{i3} is also proven.
\end{proof}

We note that $B=B(G;M_i; \pi_i)_{i\in I}$ and
$B'=B(G;M_i;\pi'_i)_{i\in I}$ are isomorphic. To see this, for
each $i\in I$ we fix $g_i\in \S(M_1)$ such that $\pi_i'=\pi_i
g_i$, and make sure that the map sending $M(i,g,j)$ in $B$ to
$g_i^{-1}M(i, g, j)g_j$ in $B'$ is an isomorphism. Further,
$$B(G;M_i; \pi_i)_{i\in I}\simeq B(\pi_1^{-1}G\pi_1;M_i;\pi_1^{-1}\pi_i)_{i\in I}.$$
Therefore, we can always choose $S$ such that $\pi_1=e$. We also
note that $B=B'$ if and only if all $\pi_i\pi_j'^{-1}$-s lie in
the normalizer of $G$ in $\S(M_1)$.

In the following theorem we give a classification of minimal transitive subsemigroups
of $\IS_{fin}(X)$ modulo the classification of minimal transitive
subgroups of the finite full symmetric groups $\S_k$.

\begin{theorem}\label{th:tr_char}
\begin{enumerate}
\item\label{ii1} Let $X$ be an infinite set. Then the semigroups
$B(G;M_i;$ $\pi_i)_{i\in I}$, where $G$ is a minimal transitive
subgroup of $\S(M_1)$, constitute the full list of minimal
transitive subsemigroups of $\IS_{fin}(X)$.
\item\label{ii2} Let $X$ be a finite set, and $|X|=n$. Then the semigroups $B(G;M_i;\pi_i)_{i\in
I}$, where $|I|>1$ and $G$ is a minimal transitive subgroup of
$\S(M_1)$; and all minimal transitive subgroups of $\S(X)$,
constitute the full list of minimal transitive subsemigroups of
$\IS(X)$.
\end{enumerate}
\end{theorem}

\begin{proof}
Suppose first that $X$ is infinite. Then $\IS_{fin}(X)$ does not
possess transitive subsemigroups which are groups. To see this, we
note that any subgroup of $\IS_{fin}(X)$ is a subgroup of $\S(M)$,
where $M$ is a finite subset of $X$ and therefore such a subgroup
acts on $M$ only, which means that it can not act transitively on
$X$. Now the first claim follows from Corollary~\ref{cor:brandt}
and Items \ref{i12} and \ref{i3} of Proposition~\ref{prop:brandt}.

To prove the second claim we apply similar arguments. The only
difference with the previous case is that for finite $X$ there are
transitive subsemigroups which are groups - those are transitive
subgroups of $S(X)$.
\end{proof}

Let us look in more detail at minimal transitive subsemigroups of
$\IS_n$, which are not groups. Consider such a subsemigroup
$B(G;M_i;\pi_i)_{i\in I}$. Since $X$ is finite, so is also $I$.
Let $|I|=k$ and suppose $I=\{1,\dots, k\}$. We know that all
$M_i$-s have the same cardinality, therefore $|M_i|=\frac{n}{k}$,
$i\in I$ (in particular, $n$ is divisible by $k$). Let us write
$B(G;M_1,\dots, M_k; \pi_1,\dots\pi_k)$ for $B(G;M_i;\pi_i)_{i\in
I}$.

\begin{theorem}\label{th:tr_is_n}
Let $k$ be a divisor of $n$.
\begin{enumerate}
\item \label{iii1} All non-zero elements of $B(G;M_1,\dots, M_k;
\pi_1,\dots\pi_k)$ have rank $\frac{n}{k}$.
\item \label{iii2} The semigroup $B(G;M_1,\dots, M_k;
\pi_1,\dots\pi_k)$  with non-zero elements of rank
$\frac{n}{k}<n$ has cardinality $k^2x+1$, where $x=|G|$. It is a
minimal transitive subsemigroup of $\IS_n$.
\item \label{iii3} Let $t(k)$ be the number of minimal transitive
subgroups of $\S_k$ up to isomorphism. Then, up to isomorphism,
the number of minimal transitive subsemigroups of $\IS_n$ equals
$$\sum_{k\text{ divides }n}t(k).
$$
\end{enumerate}
\end{theorem}

\begin{proof} By Theorem~\ref{th:tr_char}, Item~\ref{ii2} we know
that $B(G;M_1,\dots, M_k; \pi_1,\dots\pi_k)$ has elements of rank
$\frac{n}{k}$. But all non-zero elements of $B(G;M_1,\dots, M_k;
\pi_1,\dots\pi_k)$ are of the same rank (see the third paragraph
of the proof of Proposition~\ref{prop:brandt}), which
proves~\ref{iii1}.

Brandt semigroups with isomorphic groups and equicardinal index
sets are isomorphic. Besides, the cardinality of a Brandt
semigroup with the index set I and group $G$ is given by
$|I|^2|G|+1$, which implies~\ref{iii2}.

The statement \ref{iii3} is obvious.
\end{proof}

We conclude this section by an example of a minimal transitive
subsemigroup of $\IS_{8}$.
\begin{example}{\em
Let $k=2$. Choose a partition of $\{1,\dots, 8\}$ into two
$4$-element subsets, for example, $\{1,2,3,4\}$, $\{5,6,7,8\}$.
Denote the first set by $M_1$, and the second one by $M_2$. Now we
choose a cyclic subgroup of order $4$ in $\S(M_1)$, for example,
we take $G=\langle (1,2,3,4)\rangle$. Choose a bijection $\pi_2:
M_1\to M_2$, for example, $\pi_2(i)=i+4$, $i=1,2,3,4$. We obtain
the semigroup $B(G; M_1, M_2; e, \pi_2)$, which has $4\cdot
2^2+1=17$ elements which are listed below using the cycle-chain
notation (see~\cite{GM} and the references therein):
$$\begin{array}{ll}(1,2,3,4)5]6]7]8];&
(1,3)(2,4)5]6]7]8];\\ (1,4,3,2)5]6]7]8];&
(1)(2)(3)(4)5]6]7]8];\\
(1,5](2,6](3,7](4,8]; &(1,6](2,7](3,8](4,5];\\
(1,7](2,8](3,5](4,6]; &(1,8](2,5](3,6](4,7];\\
(5,1](6,2](7,3](8,4]; &(6,1](7,2](8,3](5,4];\\
(7,1](8,2](5,3](6,4]; &(8,1](5,2](6,3](7,4];\\
(5,6,7,8)1]2]3]4];& (5,7)(6,8)1]2]3]4];\\
(5,8,7,6)1]2]3]4];& (5)(6)(7)(8)1]2]3]4]\\
\text{ and } 0.
\end{array}
$$}
\end{example}
\section{Classification of semitransitive sub\-se\-mi\-groups of $\IS_n$ of
the mi\-ni\-mal car\-di\-na\-li\-ty}
%}

In this section we switch to semitransitive subsemigroups of
$\IS_{fin}(X)$. As any transitive subsemigroup is automatically
semitransitive, and transitive subsemigroups have been dealt with
in the previous section, we can limit our attention to
semitransitive subsemigroups of $\IS_{fin}(X)$, which are not
transitive.

Let $S$ be a semigroup with the zero element $0$. A non-zero
element $a\in S$ is called {\em nilpotent} provided that some
power of $a$ equals $0$.

In what follows let $0$ stand for the zero element of
$\IS_{fin}(X)$, that is $0$ is the nowhere defined partial
permutation of $X$.

\begin{proposition}\label{pr:semitr_zero_nilp}
A semitransitive subsemigroup of $\IS_{fin}(X)$, that is not
transitive, contains $0$ and a nilpotent element.
\end{proposition}

\begin{proof}
Let $S$ be a semitransitive subsemigroup of
$\IS_{fin}(X)$ that is not transitive. Then $S$ contains an element of a positive rank
strictly less than $|X|$. This is obvious for infinite $X$.
Suppose $X$ is finite, and $S$ contains only elements of rank
$|X|$ and possibly $0$. Then $S$ is a subsemigroup of $\IS(X)$ with
possibly adjoint $0$. But a subsemigroup of a finite group is in
fact a group. Then $S=H$ or $S=H\cup \{0\}$ where $H$ is a
semitransitive subgroup of $\IS(X)$. This implies that $S$ is
transitive, which contradicts the assumption.

Suppose $S$ does not contain $0$. Let $\varphi\in S$ be a non-zero
element and $\rank\varphi<|X|$. Then some power of $\varphi$ is a
non-zero idempotent, which also belongs to $S$ and has rank less
or equal than that of $\varphi$. Let $\psi\in S$ be a non-zero
idempotent with the minimum possible rank (among all elements of
$S$). Let $Y=\dom\psi=\ran\psi$. Set $Z=X\setminus Y$. Since both
$Y$ and $Z$ are non-empty, we can choose some $i\in Y$ and $j\in
Z$. Let $\alpha$ be the element in $S$ which maps either $i$ to
$j$ or vice versa. Then $0<\rank(\psi\alpha\psi)<\rank\psi$. The
obtained contradiction shows, that $S$ must contain $0$.

Finally, we show that $S$ contains a nilpotent element. Let $\psi$
and $\alpha$ be as in the preceding paragraph. By
construction, we have $Y\cap\dom\alpha\neq\varnothing$ or
$Y\cap\ran\alpha\neq\varnothing$, implying that at least one of
the elements $\alpha\psi$ or $\psi\alpha$ is non-zero. Suppose
$\alpha\psi\neq 0$ (the other case is treated similarly). Since
$\rank(\psi\alpha\psi)<\rank\psi$, it follows that
$\psi\alpha\psi=0$. Then
$(\alpha\psi)^2=\alpha(\psi\alpha\psi)=\alpha\cdot 0=0,$ ensuring
that $\alpha\psi$ is a nilpotent element.
\end{proof}

Let $S$ be a subsemigroup of $\IS(X)$. An element $i\in X$ is
called {\em cyclic with respect to $S$} provided that for every
$j\in X$ there is $\varphi_j\in S$ such that $i\varphi_j=j$. In
particular, if $S$ is transitive then every $i\in X$ is cyclic.

We choose $i\in X$. Then $iS^1$ is the set $\{i\}\cup iS$, that is, the set
of all those elements of $X$ where $i$ can be mapped by partial
permutations from $S$ and by the identity map. For $i,j\in X$ set
$i\leq_r j$ if there is $\varphi\in S$ such that
$i\varphi=j$. The latter is true if and only if the inclusion
$iS^1\supseteq jS^1$ holds. If $S$ is a semitransitive
subsemigroup then the relation $\leq_r$ is a linear preorder on
$X$. This implies that the relation $r$ on $X$ defined via $i r j$
if and only if $i\leq_r j$ and $j\leq_r i$ is an equivalence
relation. Moreover, the equivalence classes are naturally linearly
ordered by the order induced by $\leq_r$ which we denote just by
$\leq$. Suppose, $X$ is finite and $M_1, \dots, M_k$ are all the
equivalence classes with respect to $r$ and $M_1>M_2>\cdots >
M_k$. In particular, $M_1$ is the set of all elements that are cyclic with respect
to $S$. The action of $S$ on each $M_i$ is transitive and for each pair $i,j$ with
$i<j$ there is no element in $S$ that maps an element of $M_j$ to an element of $M_i$.
Therefore, for each $x\in M_i$ and $y\in M_j$ there is an element $\varphi$ in $S$
such that $y=x\varphi$ by semitransitivity.

\begin{example}{\em If $X$ is infinite, there are semitransitive subsemigroups
of $\IS_{fin}(X)$ without cyclic elements. For instance, take $X={\mathbb
Z}$. For every $i,j\in X$ with $i\leq j$ let $\varphi_{i,j}$ be
the element of rank $1$ which maps $i$ to $j$. All the elements
$\varphi_{i,j}$, together with $0$, form a semitransitive
subsemigroup of $\IS_{fin}(X)$. However, with respect to this
semigroup there is no cyclic element in $X$.}\end{example}

By the reason above in what follows we restrict our attention to
the case when the set $X$ is finite. We assume that $X=\{1,2\dots,
n\}$  and write $\IS_n$ for $\IS(X)$.

The discussion above leads to the following lower bound on the
cardinality of a semitransitive subsemigroup of $\IS_n$.

\begin{proposition}
The cardinality of a semitransitive but not transitive
subsemigroup of $\IS_n$ is greater than or equal to $n+1$.
\end{proposition}

\begin{proof}
Let $S$ be a semitransitive (and not transitive) subsemigroup of
$\IS_n$, and $i\in X$ an element that is cyclic with respect to $S$.
Then $S$ must contain elements that send $i$ to $1,2,\dots, n$.
Besides, $S$ contains $0$ by
Proposition~\ref{pr:semitr_zero_nilp}. It follows that the
cardinality of $S$ should be at least $n+1$.
\end{proof}

Now, our goal is to provide a construction which gives examples of
semitransitive and not transitive subsemigroups of $\IS_n$ of
cardinality $n+1$, and then to prove that every such a
subsemigroup is equal to one already constructed.

Let $k>1$ be a divisor of $n$. Consider a partition
$X=M_1\cup\dots\cup M_k$, such that $M_i\cap M_j=\varnothing$
whenever $i\neq j$, and $\left| M_i \right| = \frac{n}{k}$ for every $i$. Suppose
that $M_i=\{a_{i,1},\dots a_{i,\frac{n}{k}}\}$. Let $G$ be some
transitive permutation group acting on $M_1$. Fixing the bijection
from $M_1$ to $M_i$, sending $a_{1,j}$ to $a_{i,j}$, $1\leq j\leq
\frac{n}{k}$, $2\leq i\leq k$, one diagonally extends the action
of $G$ to the whole $X$. Consider the chain $(a_{1,1}, a_{2,1}
\dots, a_{k,1}]\in \IS(\{a_{1,1},\dots, a_{k,1}\})$. It generates
a nilpotent semigroup $T$ consisting of $k+1$ elements. We extend
the action of $T$ to $X$ also diagonally using the bijections
sending $a_{j,1}$ to $a_{j,i}$, $2\leq i\leq k$, $1\leq j\leq
\frac{n}{k}$. Consider the subsemigroup $S$ of $\IS_n$ generated
by $G$ and $T^1$. Obviously, $S$ is semitransitive. Since the
actions of $G$ and $T^1$ on $X$ commute, and no nonzero $s\in T$
acts in the same way as some $gs_1$, $g\in G, g\neq e$, $s_1\in
T^1$, it follows that $S$ is equal to the Rees factor $(G\times
T^1)/I$, where the ideal $I$ consists of all elements $(g,0)$,
$g\in G$.

\begin{example}{\em Let $n=8$ and $k=4$. We choose $M_1=\{1,2\}$, $M_2=\{3,4\}$,
$M_3=\{5,6\}$ and $M_4=\{7,8\}$. Suppose that $G=\{e,g\}$, where
$g=(1,2)(3,4)(5,6)(7,8)$, and $T=\{\varphi, \varphi^2,\varphi^3,0\}$,
where $\varphi=(1,3,5,7](2,4,6,8]$. The semigroup $(G\times T^1)/I$
consists of the following elements:
$$\begin{array}{ll}
(e,1)=(1)(2)(3)(4)(5)(6)(7)(8);& (g,1)=(1,2)(3,4)(5,6)(7,8)\\
(e,\varphi)=(1,3,5,7](2,4,6,8];& (g,\varphi)=(1,4,5,8](2,3,6,7]\\
(e,\varphi^2)=(1,5](3,7](2,6](4,8];& (g,\varphi^2)=(1,6](2,5](3,8](4,7]\\
(e,\varphi^3)=(1,7](2,8]3]4]5]6]:& (g,\varphi^3)=(1,8](2,7]3]4]5]6]\\
\text{ and } 0.
\end{array}
$$}
\end{example}

\begin{theorem}\label{th_classif_semitr_min_card}
Let $S$ be a semitransitive but not transitive subsemigroup of
$\IS_n$ of cardinality $n+1$. Then the action of $S$ on $X$
coincides with the action of some $(G\times T^1)/I$ given in the
construction above.
\end{theorem}

We prove Theorem~\ref{th_classif_semitr_min_card} in several
steps. First of all, fix a semitransitive (but not transitive)
subsemigroup of $\IS_n$. We consider the relation $r$ on $X$, and
the order $M_1>\cdots
>M_k$ on the $r$-classes. Note that $k>1$ since $S$ is not transitive.
Denote the cardinality of $M_i$
by $m_i$, $1\leq i\leq k$.

\begin{lemma}\label{lem_dom_ran} For every non-zero $\varphi \in
S$ we have the inclusions $M_1\subseteq \dom(\varphi)$,
$M_k\subseteq \ran(\varphi)$.
\end{lemma}

\begin{proof}
Let $\varphi\in S$ and $i,j\in X$. We will say that $\varphi$ {\it
has an arrow} from $i$ to $j$ provided that $i\in\dom(\varphi)$
and $i\varphi=j$. Every element of $M_1$ is cyclic, thus $S$
contains at least $n$ arrows from $i$ in total for every $i\in
M_1$. Therefore there are at least $m_1\cdot n$ arrows from the
elements of $M_1$ if we run through all elements of $S$. But an
element in $S$ contains at most $m_1$ arrows from $M_1$, which
implies that there are at least $\frac{m_1\cdot n}{m_1}=n$
elements in $S$ which have some arrows from $M_1$. Since $S$ has
precisely $n$ non-zero elements ($S$ does have the zero by
Proposition~\ref{pr:semitr_zero_nilp}), it follows that every
non-zero element in $S$ should have $m_1$ arrows from $M_1$. This
means that $M_1\subseteq \dom(\varphi)$ for every $\varphi \in S$.
The second inclusion is established in the same fashion.
\end{proof}

\begin{lemma}\label{lem_inv_action} Let $\varphi\in S$. Then
either $M_1 \varphi =M_1$ and $M_k\varphi= M_k$ or $
M_1\cap\ran(\varphi)=\varnothing$ and
$M_k\cap\dom(\varphi)=\varnothing$. Moreover, $M_1 \varphi =M_1$
holds if and only if $M_k\varphi= M_k$ holds, and
$M_1\cap\ran(\varphi)=\varnothing$ holds if and only if
$M_k\cap\dom(\varphi)=\varnothing$ holds.
\end{lemma}

\begin{proof}
Suppose that $M_1\varphi\neq M_1$ and that
$M_1\cap\ran(\varphi)\neq\varnothing$. Then there are $i,j\in M_1$
such that $\varphi$ has an arrow from $i$ to $j$. If
$i\in\dom(\varphi^t)$ for all $t$ then there is an $h$ in $M_1$
such that $f=h\varphi\notin M_1$. Then the cycle-chain
decomposition of $\varphi$ has a cycle $(i,j,\ldots)$ and a chain
$(\ldots,h,f,\ldots]$. Otherwise,
%we may assume that the arrow from $j$ goes to some
%$g\in M_s$ with $s>1$. Then
the cycle-chain decomposition of $\varphi$ has a chain
$(i,j,\ldots,g,\ldots]$, where $g\in M_s$ with $s>1$. This means
that there is some power $\varphi^l$ such that either $h$ or $j$
does not belong to its domain while $i$ still does. This
contradicts Lemma~\ref{lem_dom_ran}. Similarly, one shows that if
$\varphi$ has an arrow from $i$ to $j$ with $i,j\in M_k$ then
$M_k\varphi=M_k$. Finally, each of the cases $M_1 \varphi =M_1$
and $M_k\cap\dom(\varphi)=\varnothing$; and $M_k\varphi= M_k$ and
$M_1\cap\ran(\varphi)=\varnothing$ is impossible, since otherwise
there would exist a power  $\varphi^l$ such that it is non-zero
with either $M_k$ not in its range, or $M_1$ not in its domain,
respectively, which again contradicts Lemma~\ref{lem_dom_ran}.
\end{proof}

\begin{lemma}\label{lem_special_elem} There is an element
$\varphi\in S$ with $\dom(\varphi)=M_1$ and $\ran(\varphi)=M_k$
(thus, in particular, $m_1=m_k$).
\end{lemma}

\begin{proof}
Let $\varphi\in S$. Suppose $\varphi$ has an arrow from $i$ to $j$
for some $i\in M_1$ and $j\in M_k$. Then Lemma \ref{lem_inv_action}
implies that $M_1\cap\ran(\varphi)=\varnothing$ and
$M_k\cap\dom(\varphi)=\varnothing$. We will show that
all arrows from elements of $M_1$ in $\varphi$ go to elements of
$M_k$.  Suppose that this is not the case and that $\varphi$ has an arrow
from $x$ to $y$ with $x\in M_1$ and $y\in M_l$ with $1<l<k$.
Consider some $\psi\in S$ such that $y\in\dom(\psi)$ and $y\psi=j$. It
exists by semitransitivity and construction of the sets $M_i$.
Since $M_k\psi\neq M_k$ we have $M_k\cap\dom(\varphi)=\varnothing$
by Lemma \ref{lem_inv_action}. Thus $j\notin\dom(\psi)$. Then $\varphi\psi\in S$ with
$i\not\in\dom(\varphi\psi)$ and $x\in \dom(\varphi\psi)$, which is
impossible by Lemma~\ref{lem_dom_ran}. Similarly, one shows that
all arrows to elements of $M_k$ in $\varphi$ go from elements of
$M_1$.

Consider an arbitrary $\varphi\in S$. In view of the previous
paragraph, we have only to consider the following two cases.

{\em Case 1.} Suppose all arrows from $M_1$ in $\varphi$ go to
$M_k$. In this case we  show that in fact $\dom(\varphi)=M_1$,
that is $\varphi$ has no other arrows but those from $M_1$. If
this were not the case, $\varphi$ would have some arrow from $x\in
M_s$ to $y\in M_t$ with $1<s\leq t<k$.
%The strict inequalities $s>1$ and $t<k$ follow from Lemma \ref{lem_inv_action}.
We choose $\psi\in S$ mapping $y$ to $M_k$. It exists by
semitransitivity. Invoking Lemma \ref{lem_inv_action} we get
$M_k\cap\dom(\psi)=\varnothing$ since $M_k\psi\neq M_k$. Now,
$\varphi\psi$ is a non-zero element in $S$ whose domain does not
contain $M_1$, which is impossible by Lemma~\ref{lem_dom_ran}.
This proves that $\varphi$ has the property we are looking for.

{\em Case 2.} Suppose there is an arrow from $M_1$ in $\varphi$
that does not go to $M_k$. Consider some arrow of $\varphi$, from
$i$ to $j$, with $i\in M_1$ and $j\in M_l$, $l<k$. There is
$\psi\in S$ with an arrow from $j$ to $M_k$. Then $\varphi\psi$
has an arrow from $M_1$ to $M_k$, which, in view of the first
paragraph of this proof, implies that all arrows from $M_1$ in
$\varphi\psi$ go to $M_k$. Now Case 1 ensures that $\varphi\psi$
is the required element.
\end{proof}

\begin{proof}[Proof of Theorem~\ref{th_classif_semitr_min_card}]
We apply induction on the number of $r$-classes $k$. Consider
first the case $k=2$. From Lemma~\ref{lem_special_elem} we know
that $m_1=m_2$. Then, from Lemma~\ref{lem_inv_action} it follows
that $S$ has elements of two different types. The first type: the
elements $\varphi$ with $\dom(\varphi)=\ran(\varphi)=X$ such that
$M_1\varphi=M_1$ and $M_2\varphi=M_2$. The second type: the
elements $\varphi$ with $\dom(\varphi)=M_1$ and $M_1\varphi=M_2$.
Since elements of the first type act transitively on $M_1$, there
are at least $m_1$ such elements. Fix some element $\varphi$ of
the second type. Multiplying it with different elements of the
first type we obtain different elements of the second type,
meaning that $S$ has at least $2m_1=n$ non-zero elements. It
follows that the cardinality of the set of the elements of the
first type is $m_1$, and the restrictions to $M_1$ of these
elements form a transitive group of permutations of $M_1$. Let $T$
be a semigroup generated by $\varphi$. We have that the action of
$S$ coincides with the action of  $(G\times T^1)/I$, where the
ideal $I$ consists of all elements $(g,0)$, $g\in G$, which
finishes the proof of the induction base.

Let now $k\geq 3$. We construct a homomorphism $\gamma$ from $S$
to some semitransitive subsemigroup of $\IS(X\setminus M_k)$ with
the $r$-classes $M_1>\dots > M_{k-1}$. Let $\varphi\in S$. If
$M_k\varphi=M_k$ then also $(X\setminus M_k)\varphi=X\setminus
M_k$. In this case we set $\varphi\gamma$ to be equal to the
restriction of $\varphi$ to $X\setminus M_k$. Otherwise we have
$M_k\cap\dom(\varphi)=\varnothing$ (by
Lemma~\ref{lem_inv_action}), meaning that in the cycle-chain
notation each element of $M_k$ occurs at the end of some non-empty
chain. We define the element $\varphi\gamma$ by erasing the last
elements of all chains in $\varphi$. This construction ensures
that $\gamma$ is homomorphic, and that $S\gamma$ is a
semitransitive subsemigroup of $\IS(X\setminus M_k)$ with the
$r$-classes $M_1>\dots > M_{k-1}$. Let us estimate its
cardinality. By Lemma~\ref{lem_special_elem} we know that there is
$\varphi\in S$ with $\dom(\varphi)=M_1$ and $\ran(\varphi)=M_k$.
Besides, we have at least $m_k$ elements $\psi$ such that
$M_k\subseteq \dom(\psi)$ (this follows from
Lemma~\ref{lem_inv_action} and that a group acting transitively on
a $s$-element set is at least of cardinality $s$). Therefore,
considering all possible products $\varphi\psi$, we make sure that
$S$ has at least $m_k$ different elements $\alpha$ with
$\dom(\alpha)=M_1$ and $\ran(\alpha)=M_k$. All these elements, as
well as the zero of $S$, are mapped by $\gamma$ to the zero of
$S\gamma$, implying that $S\gamma$ has at most $n+1-m_k=m_1+\cdots
+m_{k-1}+1=|X\setminus M_k|+1$ elements. It follows that $S\gamma$
is a semitransitive subsemigroup of $\IS(X\setminus M_k)$ of the
minimum cardinality, and the induction assumption can be implied.
It follows that $m_1=\cdots =m_k$, that the action of $S\gamma$ on
$X\setminus M_k$ coincides with the action of  some $(G\times
T^1)/I$, where $G$ is a $m_1$-element group acting transitively on
$M_1$, $T$ is generated by some nilpotent element $\varphi$ all
whose chains are of length $k-1$ and go from $M_1$ through $M_2$,
$\dots$ to $M_{k-1}$. Let $\psi$ be some $\gamma$-preimage of
$\varphi$, and $\overline{T}$ the semigroup generated by it.
Then $S$ contains the subsemigroup $(G\times
\overline{T}^1)/\overline{I}$ with the ideal $\overline{I}$
consisting of all elements $(g,0)$, $g\in G$ ($0$ is the zero of
$S$). Since this semigroup and $S$ are of the same cardinality
$n+1$, we conclude that they must coincide. This completes the
proof.
\end{proof}

The assumption of minimal cardinality in Theorem \ref{th_classif_semitr_min_card}
implies that the cardinality of the sets $M_i$ are all equal. If we only assume
minimality then this is no longer the case as is shown by the following example.

\begin{example}\label{ex:nonminimal_card}
Assume that $X=\{1,2,3\}$ and that $M_1=\{1,2\}$ and $M_2=\{3\}$.
Then $S=\{(1,2)(3),(1)(2)(3),(1,3]2],(2,3]1],0\}$ is a minimal semitransitive
subsemigroup of $\IS(X)$. Its cardinality is equal to $5$, which is greater then $n+1=4$.
\end{example}

\bigskip \bigskip

\noindent{K. Cvetko-Vah, D. Kokol Bukov\v sek, T. Ko\v sir: Department of Mathematics,
University of Ljubljana, Jadranska 19, SI-1000 Ljubljana, Slovenia.\\ e-mail: karin.cvetko@fmf.uni-lj.si,
damjana.kokol@fmf.uni-lj.si,\\ tomaz.kosir@fmf.uni-lj.si.}

\medskip

\noindent{G. Kudryavtseva, Y. Lavrenyuk and A. Oliynyk: Department of Mechanics and Mathematics,
Kyiv Taras Shevchenko University, Volodymyrs'ka str. 64, 01033 Kyiv, Ukraine.\\ e-mail: akudr@univ.kiev.ua,
ylavrenyuk@univ.kiev.ua, olijnyk@univ.kiev.ua.}


\begin{thebibliography}{99}

\bibitem{BDKKO} J.~Bernik, R. Drnov\v sek, D. Kokol Bukov\v sek,
T. Ko\v sir, and M. Omladi\v c. \emph{Reducibility and triangularizability of semitransitive operator spaces}.
To appear in Houston Journal of Mathematics.

\bibitem{BDKKOR} J.~Bernik, R. Drnov\v sek, D. Kokol Bukov\v sek,
T. Ko\v sir, M. Omladi\v c, and H.~Radjavi.
\emph{On semitransitive Jordan algebras of matrices}. Preprint.

\bibitem{BGMRT} J.~Bernik, L.~Grunenfelder, M.~Mastnak, H.~Radjavi, and V.~G.~Troitsky.
\emph{On semitransitive collections of operators}. Semigroup Forum {\bf 70} (2005), 436--450.

\bibitem{GM} O.Ganyushkin and V. Mazorchuk. \emph{Introduction to classical finite transformation
semigroups}. Preprint, available on-line at
http://www.math.uu.se/\%7Emazor/PREPRINTS/SEMI/book.pdf

\bibitem{hig} P. Higgins. \emph{Techniques of Semigroup Theory}. Oxford University Press,
Oxford, 1992.

  \bibitem{hulthesis}
A. Hulpke. \emph{Konstruktion transitiver Permutationgruppen}. Dissertation RWTH Aachen.
Aachener Beitr\"age zur Mathematik, Band 18, 1996. Available on-line at http://www.math.colostate.edu/~hulpke/publ.html

   \bibitem{hul}
A. Hulpke. \emph{Constructing Transitive Permutation Groups}. J. Symbolic Comput. {\bf 39} (2005), 1--30.
Available on-line at http://www.math.colostate.edu/~hulpke/publ.html

  \bibitem{kop1}
T. I. Kopylova. \emph{Solvable minimal transitive groups of permutations of degree $pq$}.
Vesci Akad. Navuk BSSR Ser. Fiz.-Mat. Navuk 1985, no. 6, 54--60.

\bibitem{kop2}
T. I. Kopylova. \emph{Minimal transitive permutation groups of finite degree}.
Vesci Akad. Navuk BSSR Ser. Fiz.-Mat. Navuk 1989, no. 2, 21–-25.

\bibitem{law} M. V. Lawson. \emph{Inverse Semigroups: The Theory of Partial Symmetries}.
World Scientific, 1998.


\bibitem{pet} M. Petrich. \emph{Inverse semigroups}. Wiley, New York, 1984.


\bibitem{RaT} H.~Radjavi and V.~G.~Troitsky. \emph{Semitransitive subspaces of operators}. To appear
in Linear and Multilinear Algebra.

\bibitem{RT} H. Rosenthal and V. G. Troitsky. \emph{Strictly semi-transitive operator algebras}.
Journal of Operator Theory  {\bf 53}  (2005), 315--329.

\bibitem{schein} B.M. Schein. \emph{Representations of generalized
groups}. Izv. Vys\u{s}. U\u{c}ebn. Zav. Matem. Issue 3 (1962),
164-176 (Russian).

\bibitem{Bled}
Semitransitivity Working Group at LAW'05, Bled. \emph{Semitransitive subspaces of matrices}.
Electronic Journal of Linear Algebra {\bf 15} (2006), 225-238.

  \bibitem{supr}
D. A. Suprunenko. \emph{Solvable minimal transitive permutation groups of degree $pq$}.
Dokl. Akad. Nauk SSSR {\bf 269} (1983), no. 2, 295--298.

  \bibitem{tai}
Igo Dak Tai. \emph{On Minimal transitive permutation groups}.
Vesci Akad. Navuk BSSR Ser. Fiz.-Mat. Navuk 1976, no. 6, 5--14.

\end{thebibliography}
\end{document}